\documentclass[10pt]{article}
\font\smallit=cmti10

\usepackage{amssymb,latexsym,amsmath,epsfig,amsthm} %% Add other packages as necessary

\makeatletter

\renewcommand\section{\@startsection {section}{1}{\z@}
{-30pt \@plus -1ex \@minus -.2ex}
{2.3ex \@plus.2ex}
{\normalfont\normalsize\bfseries\boldmath}}

\renewcommand\subsection{\@startsection{subsection}{2}{\z@}
{-3.25ex\@plus -1ex \@minus -.2ex}
{1.5ex \@plus .2ex}
{\normalfont\normalsize\bfseries\boldmath}}

\renewcommand{\@seccntformat}[1]{\csname the#1\endcsname. }

\makeatother

\newtheorem{theorem}{Theorem}
\newtheorem{lemma}{Lemma}
\newtheorem{prop}{Proposition}
\newtheorem{cor}{Corollary}

\theoremstyle{definition}
\newtheorem{definition}{Definition}

\newtheorem{remark}{Remark}
\newtheorem{example}{Example}

\usepackage{graphicx}
\usepackage[colorlinks=true,citecolor=black,linkcolor=black,urlcolor=magenta]{hyperref}
\usepackage{ulem} 
\usepackage{url} 
\usepackage{calc} 
\usepackage{MnSymbol}

\newcommand{\citep}{\cite}
\newcommand{\cf}[0]{\textit{cf}.\ } 
\renewcommand{\emph}[1]{\textit{#1}} 

\newcommand{\gkpSI}[2]{\ensuremath{\genfrac{\lbrack}{\rbrack}{0pt}{}{#1}{#2}}} 
\newcommand{\gkpSII}[2]{\ensuremath{\genfrac{\lbrace}{\rbrace}{0pt}{}{#1}{#2}}}

\usepackage{enumitem}
\setlist[itemize]{leftmargin=0.35in}

\newcommand{\tagsleft}[0]{\let\veqno\@@leqno}

\usepackage{rotating} 

\usepackage{diagbox}
\newcommand{\trianglenk}[2]{$\diagbox{#1}{#2}$}

\usepackage{ifthen}
\newcommand{\Hn}[2]{
     \ifthenelse{\equal{#2}{1}}{H_{#1}}{H_{#1}^{\left(#2\right)}}
}

\DeclareMathOperator{\Sum}{Sum} 

\allowdisplaybreaks

\begin{document}

\begin{center}
\uppercase{\bf Combinatorial Sums and Identities Involving Generalized Sum-of-Divisors Functions with  
               Bounded Divisors}
\vskip 20pt
{\bf Maxie Dion Schmidt} \\
{\smallit School of Mathematics, Georgia Institute of Technology, Atlanta, GA 30332, USA} \\
{\tt mschmidt34@gatech.edu, maxieds@gmail.com}
\end{center}
\vskip 20pt
\centerline{\smallit Received: , Revised: , Accepted: , Published: } % We will fill in the dates
\vskip 30pt

\centerline{\bf Abstract}
\noindent
The class of Lambert series generating functions (LGFs) denoted by  
$L_{\alpha}(q)$ formally enumerate the 
generalized sum-of-divisors functions, $\sigma_{\alpha}(n) = \sum_{d|n} d^{\alpha}$, 
for all integers $n \geq 1$ and fixed real-valued parameters $\alpha \geq 0$. 
We prove new formulas expanding the higher-order derivatives of these LGFs. 
The results we obtain are combined to express new identities expanding the 
generalized sum-of-divisors functions. These new identities are expanded in the form of 
sums of polynomially scaled multiples of a related class of divisor sums depending on 
$n$ and $\alpha$. 

%\pagestyle{myheadings} 
%\markright{\smalltt INTEGERS: 20 (2020)\hfill} 
%\thispagestyle{empty} 
%\baselineskip=12.875pt 
\vskip 30pt

\section{Introduction} 

\subsection{Generating the generalized sum-of-divisors functions} 

For any indeterminate $q \in \mathbb{C}$ 
satisfying $|q| < 1$ and any artihmetic function $f$, we have that 
\[
[q^n]\left(\sum_{i=1}^{n} \frac{f(i) q^i}{1-q^i}\right) = \sum_{d|n} f(d), 
\]
where the right-hand-side sum is indexed over all divisors $d$ of the $n \geq 1$. 
The identity in the previous equation is correct because whenever $|q| < 1$, the 
next infinite series converges absolutely, and the following expansions are equivalent: 
\begin{align*} 
\sum_{i \geq 1} \frac{f(i) q^i}{1-q^i} & = \sum_{m \geq 1} \sum_{i \geq 1} f(i) q^{mi} 
     = \sum_{n \geq 1} \left[\sum_{mi=n} f(i)\right] q^n. 
\end{align*} 
We treat the Lambert series, $L_{\alpha}(q)$, corresponding to the special case where 
$f(n) := n^{\alpha}$ for $\alpha \in \mathbb{R}$ 
formally in our context as the generating functions of 
an important sequence of classical number theoretic functions defined by the divisor sums 
$\sigma_{\alpha}(n) := \sum_{d|n} d^{\alpha}$ \cite{HARDY-WRIGHT}. 
In particular, for any fixed $\alpha \in \mathbb{R}$, we have that \citep[\S 27.7]{NISTHB} 
\begin{equation} 
\label{eqn_LAlphaq_LSeries_exp_def_v1}
L_{\alpha}(q) := \sum_{n \geq 1} \frac{n^{\alpha} q^n}{1-q^n} = 
     \sum_{m \geq 1} \sigma_{\alpha}(m) q^m, |q| < 1. 
\end{equation}
In this article we are interested in proving new properties of the functions 
$\sigma_{\alpha}(n)$, for real $\alpha \geq 0$, through algebraic operations on and 
new combinatorially motivated identities of the generating functions $L_{\alpha}(q)$. 
Our new approach is to use identities expanding the higher order derivatives of the 
Lambert series generating functions of the \emph{generalized sum-of-divisors functions}, 
$\sigma_{\alpha}(n)$, to prove new identities satisfied by these functions. 
Note that when $\alpha < 0$, by symmetry we can recover formulas for the 
generalized sum-of-divisors functions as 
$\sigma_{-\alpha}(n) = \sigma_{\alpha}(n) \cdot n^{-\alpha}$. 

\subsubsection{Comparisons to known convolution formulas} 

The special cases given by the 
\emph{divisor function}, $d(n) \equiv \sigma_0(n)$, and the ordinary classical 
\emph{sum-of-divisors function}, $\sigma(n) \equiv \sigma_1(n)$, are of much 
interest in modern and traditional number theory.
Some recent generating-function-based approaches to enumerating identities for divisor 
functions are found in the references \citep{SEGLADUN-GFS,JNT-ANEWLOOK}. 
There are several known forms of divisor sum convolution identities 
of the form 
\begin{align} 
\label{eqn_CvlType_Idents_for_SODFuncs} 
\sigma_{\alpha}(n) & = c_1\left((a_1n+a_2) \sigma_{\beta}(n) + a_3 
     \sigma_{\gamma}(n) + a_4 \times \sum_{k=1}^{n-1} 
     \sigma_{\gamma}(k) \sigma_{\beta}(n-k) 
     \right), 
\end{align} 
including those found in the references for the triples 
$$(\alpha, \beta, \gamma) \in \left\{ (3,1,1), (5,3,1), (7,5,1), (9,7,1), (9,5,3) \right\}.$$ 
The special case convolution formulas 
result from identities and functional equations satisfied by 
Eisenstein series in the context of modular forms 
\citep{ELEMEVALCVLSUMS2002,ONARITHFNS-RAMANUJAN}. 
For example, the following two convolution identities are known relating special cases of the 
generalized sum-of-divisors functions: 
\begin{align} 
\label{eqn_cvl_idents_examples} 
\sigma_3(n) & = \frac{1}{5}\left(6n \cdot \sigma_1(n) - \sigma_1(n) + 12 \times 
     \sum_{k=1}^{n-1} \sigma_1(k) \sigma_1(n-k)\right) \\ 
\notag 
\sigma_5(n) & = \frac{1}{21}\left(10(3n-1) \cdot \sigma_3(n) + \sigma_1(n) + 
     240 \times \sum_{k=1}^{n-1} \sigma_1(k) \sigma_3(n-k)\right). 
\end{align} 
Other convolution identities related to the divisor function and sum-of-divisors 
function are proved in \citep{JNT-NEWCVLSDN}. 
In contrast, we focus on proving new expansions involving 
sums of polynomial multiples of 
generalized divisor sums and so-termed bounded-index divisor functions in place of the 
discrete convolutions of these functions in the forms of the identities cited in 
\eqref{eqn_CvlType_Idents_for_SODFuncs} and \eqref{eqn_cvl_idents_examples}. 

\subsection{Natural interpretations by bounded-index divisor functions} 

The new identities and closed-form expansions we derive in this article shed some new 
light on how we can reconcile more combinatorial expansions of the 
series in \eqref{eqn_LAlphaq_LSeries_exp_def_v1} 
with properties of the so-called \emph{bounded-index divisor functions} 
defined in \eqref{eqn_modified_divisorFns_with_bdd_divs} and 
\eqref{eqn_BoundedIndex_DivisorFuncs_def_v2} below. 
Namely, we define the following variants of the generalized sum-of-divisors functions 
for any positive integers $n,k,m \geq 1$: 
\begin{align} 
\label{eqn_modified_divisorFns_with_bdd_divs} 
B_{k,m}(\alpha; n) := 
     \sum_{\substack{d|n \\ d \leq \left\lfloor \frac{n}{m} \right\rfloor}} 
     \binom{\frac{n}{d}-m+k}{k} d^{\alpha}. 
\end{align} 
It turns out that the definitions of these modified, or binomially scaled 
bounded-index, versions of the classical divisor functions $\sigma_{\alpha}(n)$, 
appear naturally in formal power series manipulations of the generating functions 
$L_{\alpha}(q)$ that generate the classical sequences. 

In particular, we can take expansions of the left-hand-side of 
\eqref{eqn_BinCvlDivFn_as_CoeffsOfLSeries_exp} below as 
geometric and binomial series in powers of $q$. These expansions 
imply that the functions $B_{k,m}(\alpha; n)$ defined by 
\eqref{eqn_modified_divisorFns_with_bdd_divs} 
correspond to the following series coefficient formulas:
\begin{align} 
\label{eqn_BinCvlDivFn_as_CoeffsOfLSeries_exp} 
[q^n] \sum_{i \geq 1} \frac{i^{\alpha} q^{mi}}{(1-q^i)^{k+1}} & = 
     \sum_{\substack{d|n \\ d \leq \left\lfloor \frac{n}{m} \right\rfloor}} 
     \binom{\frac{n}{d}-m+k}{k} d^{\alpha},\ 
     \alpha \geq 0, m \in \mathbb{Z}^{+}, k \in \mathbb{N}. 
\end{align} 
Thus the definitions of the parameterized sequences in 
\eqref{eqn_modified_divisorFns_with_bdd_divs} lead us to a more natural definition for 
bounded-index variants of the classical sum-of-divisors functions. 
To be precise, we will find formulas based on 
\eqref{eqn_BinCvlDivFn_as_CoeffsOfLSeries_exp} 
for the $\sigma_{\alpha}(n)$ expressed as 
quasi-polynomially scaled combinations of the particular variants of the 
bounded-index divisor functions defined in the following equation:
\begin{align}
\label{eqn_BoundedIndex_DivisorFuncs_def_v2} 
\sigma_{\alpha,m}(n) := \sum_{\substack{d|n \\ d \leq \left\lfloor \frac{n}{m} \right\rfloor}} 
     d^{\alpha}, n \geq 1; 1 \leq m \leq n; \alpha \in \mathbb{R}. 
\end{align} 
To state the next results, we adopt the double-indexed bracket notation for the unsigned 
Stirling numbers of the first and second kinds. 
Other common notation for the Stirling number triangles is given in 
\cite[\S 26.8]{NISTHB} as 
$\gkpSI{n}{k} = (-1)^{n-k} s(n,k)$ and 
$\gkpSII{n}{k} = S(n,k)$ for non-negative integers $n, k \geq 0$ such that 
$0 \leq k \leq n$. 
Then by the binomial theorem and the known identity by which we can expand the 
single factorial function in terms of the Stirling numbers of the first kind as 
\cite[\cf \S 6.1]{GKP} 
\[
n! = \sum_{m=0}^{n} \gkpSI{n}{m} (-1)^{n-m} n^m, \mathrm{\ for\ } n \geq 0, 
\]
we also easily prove that \cite[\S 6.1]{GKP} 
\begin{align*} 
\sum_{\substack{d|n \\ d \leq \left\lfloor \frac{n}{m} \right\rfloor}} & 
     \binom{\frac{n}{d}-m+k}{k} d^{\alpha} \\ 
     & = 
     \sum_{\substack{d|n \\ d \leq \left\lfloor \frac{n}{m} \right\rfloor}} \left[ 
     \sum_{j=0}^{k+1} \sum_{r=0}^{j-1} \gkpSI{k+1}{j} \binom{j-1}{r} \frac{m^{j-1-r} \cdot 
     (-1)^{j-1-r}}{k!} \times 
     \left(\frac{n}{d}\right)^r d^{\alpha} 
     \right] \\ 
     & = \sum_{r=0}^{k} \sum_{j=r}^{k} \gkpSI{k+1}{j+1} \binom{j}{r} \frac{m^{j-r} \cdot 
     (-1)^{j-r}}{k!} \times 
     n^r \cdot \sigma_{\alpha-r,m}(n). 
\end{align*} 
The expansions in the previous two equations then motivate our uses of the functions 
$\sigma_{\alpha,m}(n)$, in expanding higher-order derivatives of 
\eqref{eqn_BinCvlDivFn_as_CoeffsOfLSeries_exp} as suggested by the identities stated in 
Lemma~\ref{lemma_formulas_for_special_series_coeffs} of the next subsection. 
For reference, a table of particular values of $\sigma_{\alpha,m}(n)$ for 
$(n, m) \in \mathbb{N} \times \mathbb{N}$, when the 
fixed symbolic parameter $\alpha$ remains unevaluated in the resulting expressions, 
is found in Table~\ref{table_bdd-divisor_divsumfns_listings}. 

\begin{table}[ht!] 
\centering\tiny 
\begin{equation*} 
\begin{array}{|c|l|l|l|l|} \hline 
\trianglenk{n}{m} & 1 & 2 & 3 & 4 \\ \hline\hline 
1 & 1 & 0 & 0 & 0 \\
2 & 1+2^{\alpha } & 1 & 0 & 0 \\
3 & 1+3^{\alpha } & 1 & 1 & 0 \\
4 & 1+2^{\alpha }+4^{\alpha } & 1+2^{\alpha } & 1 & 1 \\
5 & 1+5^{\alpha } & 1 & 1 & 1 \\
6 & 1+2^{\alpha }+3^{\alpha }+6^{\alpha } & 1+2^{\alpha }+3^{\alpha } & 1+2^{\alpha } & 1 \\
7 & 1+7^{\alpha } & 1 & 1 & 1 \\
8 & 1+2^{\alpha }+4^{\alpha }+8^{\alpha } & 1+2^{\alpha }+4^{\alpha } & 1+2^{\alpha } & 1+2^{\alpha } \\
9 & 1+3^{\alpha }+9^{\alpha } & 1+3^{\alpha } & 1+3^{\alpha } & 1 \\
10 & 1+2^{\alpha }+5^{\alpha }+10^{\alpha } & 1+2^{\alpha }+5^{\alpha } & 1+2^{\alpha } & 1+2^{\alpha } \\
11 & 1+11^{\alpha } & 1 & 1 & 1 \\
12 & 1+2^{\alpha }+3^{\alpha }+4^{\alpha }+6^{\alpha }+12^{\alpha } & 1+2^{\alpha }+3^{\alpha }+4^{\alpha }+6^{\alpha } & 1+2^{\alpha }+3^{\alpha }+4^{\alpha
   } & 1+2^{\alpha }+3^{\alpha } \\
13 & 1+13^{\alpha } & 1 & 1 & 1 \\
14 & 1+2^{\alpha }+7^{\alpha }+14^{\alpha } & 1+2^{\alpha }+7^{\alpha } & 1+2^{\alpha } & 1+2^{\alpha } \\
15 & 1+3^{\alpha }+5^{\alpha }+15^{\alpha } & 1+3^{\alpha }+5^{\alpha } & 1+3^{\alpha }+5^{\alpha } & 1+3^{\alpha } \\
16 & 1+2^{\alpha }+4^{\alpha }+8^{\alpha }+16^{\alpha } & 1+2^{\alpha }+4^{\alpha }+8^{\alpha } & 1+2^{\alpha }+4^{\alpha } & 1+2^{\alpha }+4^{\alpha } \\
\hline\hline 
\end{array}
\end{equation*} 

\caption{The bounded-index divisor sum functions, $\sigma_{\alpha,m}(n)$} 
\label{table_bdd-divisor_divsumfns_listings} 

\end{table} 

\subsection{Combinatorial lemmas expanding series coefficients of LGFs} 

We will prove the next few results rigorously in 
Section~\ref{Section_appendix_proofs_of_lemmas}. 
For now, we will motivate how we derived 
the more technical multiple summation identities for the 
sum-of-divisor functions, which we will state as 
main theorems later in the article. 

\begin{lemma}[A pair of utility sums]  
\label{lemma_SumOverjPowqj_SumIdents_v2} 
For any fixed non-zero $q$ and integers $p \geq 0$ and $n \geq p$, we have the 
following two identities:
\begin{align*} 
\tag{i} 
\sum_{j=0}^n \frac{j!}{(j-p)!} q^j & = \frac{1}{(1-q)^{p+1}}\left( 
     p! \cdot q^p + \sum_{k=0}^p \binom{p}{k} \frac{(-1)^{k+1} (n+1)! q^{n+k+1}}{ 
     (n-p)! (n+1-p+k)}\right) \\ 
\tag{ii} 
\sum_{j=0}^n \frac{(j+1)!}{(j+1-p)!} q^j & = \frac{1}{(1-q)^{p+1}}\left( 
     p! \cdot q^{p-1} + \sum_{k=0}^p \binom{p}{k} \frac{(-1)^{k+1} (n+2)! q^{n+k+1}}{ 
     (n+1-p)! (n+2-p+k)}\right). 
\end{align*} 
\end{lemma} 

\begin{lemma}[Formulas for special series coefficients] 
\label{lemma_formulas_for_special_series_coeffs} 
For any fixed $\alpha \geq 0$ and integers $s \geq 0$, we have each of the 
following series coefficient formulas: 
\begin{align*} 
\tag{i} 
[q^x] q^s D^{(s)}\left[\frac{L_{\alpha}(q)}{1-q}\right] 
     &  = 
     \sum_{r=0}^s \sum_{k=1}^x 
     \binom{s}{r} \binom{x-k}{s-r} \frac{(s-r)! k!}{(k-r)!} \sigma_{\alpha}(k) \\ 
\tag{ii} 
[q^x] q^s D^{(s)}\left[\sum_{i \geq 1} \frac{i^{\alpha-1} q^i}{(1-q)^2} 
     \right] 
     &  = 
     \sum_{r=0}^s \sum_{k=0}^x 
     \binom{s}{r} \binom{x-k+1}{s-r+1} \frac{(s-r+1)! k^{\alpha-1} \cdot k!}{ 
     (k-r)!} \\ 
\tag{iii} 
[q^x] q^s D^{(s)}\left[\sum_{i \geq 1} \frac{i^{\alpha-1} q^i}{1-q} 
     \right] 
     &  = 
     \sum_{r=0}^s \sum_{k=0}^x 
     \binom{s}{r} \binom{x-k}{s-r} \frac{(s-r)! k^{\alpha-1} \cdot k!}{ 
     (k-r)!}. 
\end{align*} 
\end{lemma} 

\begin{lemma}[Higher-order derivatives of Lambert series] 
\label{lemma_sthDerivsOfLambertSeries} 
For any fixed non-zero $q \in \mathbb{C}$ such that $|q| < 1$, $n \in \mathbb{Z}^{+}$, and 
integer $s \geq 0$, we have the following results: 
\begin{align} 
\tag{i} 
q^s D^{(s)}\left[\frac{q^n}{1-q^n}\right] & = 
     \sum_{m=0}^s \sum_{k=0}^m \gkpSI{s}{m} \gkpSII{m}{k} 
     \frac{(-1)^{s-k} k! \cdot n^m \cdot q^n}{(1-q^n)^{k+1}} \\ 
\tag{ii} 
q^s D^{(s)}\left[\frac{q^n}{1-q^n}\right] & = 
     \sum_{r=0}^s \left( 
     \sum_{m=0}^s \sum_{k=0}^m \gkpSI{s}{m} \gkpSII{m}{k} \binom{s-k}{r} 
     \frac{(-1)^{s-k-r} k! \cdot n^m}{(1-q^n)^{k+1}} 
     \right) q^{(r+1)n}. 
\end{align} 
\end{lemma} 

\subsection{Motivating new identities for 
            Lambert series generating functions} 

We can generate the generalized 
sum-of-divisors functions by considering only the terms in the following truncated 
series identities (see Lemma~\ref{lemma_LambertSeriesTerm_Sumjim2_Exps}): 
\begin{align} 
\notag 
\sigma_{\alpha}(n) & = [q^n] \sum_{i=1}^n \frac{i^{\alpha} q^i}{1-q^i} \\ 
\label{eqn_key_LSeries_exp_ident} 
     & = 
     [q^n]\left(\sum_{i=1}^n \left[\frac{i^{\alpha-1} q^i}{1-q} + 
     \sum_{j=0}^{i-2} \frac{i^{\alpha-1} q^{i+j} (i-1-j) (1-q)}{(1-q^i)}\right]  
     \right). 
\end{align} 
In this form, we see that we may evaluate the inner sum exactly as 
\begin{align*} 
\sum_{j=0}^{i-2} q^{i+j} (i-1-j) (1-q) & = 
     \frac{q^{2i}}{1-q} + \frac{(i-1)q^i-i q^{i+1}}{1-q}. 
\end{align*} 
We then have an immediate corollary to an identity involving the bounded divisor sum functions 
from \eqref{eqn_modified_divisorFns_with_bdd_divs}. We can see that 
\begin{align*} 
\sigma_{\alpha}(n) & = \sum_{k=1}^n \left[k^{\alpha-1} + \sigma_{\alpha-1,2}(k) + 
     \sigma_{\alpha}(k) - \sigma_{\alpha-1}(k)-\sigma_{\alpha}(k-1) 
     \right]. 
\end{align*} 
It is clear that $\sigma_{\alpha,2}(n) = \sigma_{\alpha}(n) - n^{\alpha}$ 
for all $n \geq 1$. In this case, we have so far not recovered any new information 
about the generalized sum-of-divisors functions $\sigma_{\alpha}(n)$. 
We can still 
continue reasoning in this way by expanding the Lambert series generating function 
partial sums using the next 
higher-order derivative identity involving the Stirling number triangles 
for any fixed integers $s \geq 1$ in the following two forms: 
\begin{align} 
\label{eqn_SNumDblSum_iPowm_sthDerivFormula} 
q^s D^{(s)}\left[\frac{q^i}{1-q^i}\right] & = 
     \sum_{m=0}^s \sum_{k=0}^m \gkpSI{s}{m} \gkpSII{m}{k} 
     \frac{(-1)^{s-k} k! \cdot i^m}{(1-q^i)^{k+1}} \\ 
\notag 
     & = \sum_{r=0}^s \left( 
     \sum_{m=0}^s \sum_{k=0}^m \gkpSI{s}{m} \gkpSII{m}{k} \binom{s-k}{r} 
     \frac{(-1)^{s-k-r} k! \cdot i^m}{(1-q^i)^{r+1}} 
     \right) q^{(r+1)i}. 
\end{align} 
Applying the identity in the previous pair of equations from 
\eqref{eqn_SNumDblSum_iPowm_sthDerivFormula} inductively, 
we can expand the expressions for $\sigma_{\alpha}(n)$ clearly by multiple summations. 

\subsubsection{Remarks} 

The bounded-index variations of the classical divisor functions, 
$\sigma_{\alpha,m}(n)$, are non-trivial objects themselves. 
That is to say, that these functions reveal at least as much deep information as the 
classical sum-of-divisors functions.  
We will see next that we can use the functions $\sigma_{\alpha,m}(n)$ 
to re-write the coefficients of the LGF-type generating functions of 
$\sigma_{\alpha}(n)$ by expansions combining the 
forms of the bounded-index functions in the spirit of 
\eqref{eqn_CvlType_Idents_for_SODFuncs}. 

We note that the initial forms of the new formulas we give next are 
at least in part motivated by experimental mathematics with summations and 
algebraic formal manipulations of polynomials. That these proofs are rigorously 
justified though an approach to this subject using a modern 
computer algebra system such as \emph{Mathematica} is in no small part 
responsible for the author's initial discovery of the new identities. 
Since the 
$N$-order accurate partial sums of the Lambert series, $L_{\alpha}(q)$, 
accurately generate the $\sigma_{\alpha}(n)$ as coefficients of an ordinary power series 
expansion for any $1 \leq n \leq N$, we are able to forego almost all 
considerations of the convergence of the $L_{\alpha}(q)$ as an analytic object. 

The key to interpreting the theorems stated and proved  
in Section~\ref{Section_DefsStmtsOfMainResults} is to abstract the 
technical nature of the resulting coefficients. 
Instead of focusing on the problem of multiple finite 
summations (which is easily conquered in light of modern 
CAS platforms and simplifications via finite summation packages), we should 
consider the ways these identities relay interesting new substructural variants 
of the generalized sum-of-divisor functions expressed by familiar combinatorial 
sequences and constructions. 
We also can apply mechanical 
summation identities to simplify nested sums involving Stirling numbers and 
hypergeometric function terms, an application that typically mitgates the initially 
complicated appearance of the coefficients involved in these formulas. 

\subsection{Characteristic examples of the new results} 

In what follows, we build up successive notation to prove the corresponding formulas 
for the coefficients of the quasi-linear combinations for the $\sigma_{\alpha}(n)$. 
It is key to keep the simplifications into more explicit direct 
expansions, like those given below in 
Example~\ref{example_dn_sigman_formula_exps_s2} and 
Example~\ref{example_Intro2} below, in mind when evaluating the significance of the 
more general results stated in the next section. 

\begin{example}[Special cases for classical divisor functions] 
\label{example_dn_sigman_formula_exps_s2} 
For the familiar explicit cases of $\alpha := 0, 1$ corresponding to the 
divisor function, $d(n)$, and the (ordinary) sum-of-divisors function, 
$\sigma(n)$, respectively, we obtain the next expansions. 
This leads to the statement of a few 
special case identities that convey the characteristic nature of the 
expansions we can expect more generally from the main theorems in 
Section~\ref{Section_DefsStmtsOfMainResults}. 
An easy corollary of what we prove in the next sections defines these classical 
cases of interest 
according to the sums 
\begin{align*} 
\binom{x}{2} d(x) & = \rho_2^{(0)}(x) + \sum_{k=1}^{x-1} \tau_{2}^{(0)}(k) \\ 
\binom{x}{2} \sigma(x) & = \rho_2^{(1)}(x) + \sum_{k=1}^{x-1} \tau_{2}^{(1)}(k), 
\end{align*} 
where the component functions in these summation-based formulas have 
explicit representations given in 
\eqref{eqn_TauRho_Idents_for_s2Alpha01} below. 
\begin{align} 
\label{eqn_TauRho_Idents_for_s2Alpha01} 
\tau_{2}^{(0)}(k) & = 
\frac{1}{4} \left((3k-2) \sigma_0(k)-(k-1) \sigma_1(k)-\sigma_2(k)\right) \\ 
\notag 
     & \phantom{=\frac{1}{4}\ } + 
     \frac{1}{4} \left(k^2 \sigma _{-2,3}(k)+k (k-3) \sigma _{-1,3}(k) - 
     (3k-2) \sigma _{0,3}(k) + 2 \sigma _{1,3}(k)\right) \\ 
\notag 
\rho_{2}^{(0)}(x) & = 
     \frac{1}{4} \left(\left(2 x^2 + x - 2\right) \sigma_0(x)- 
     (x-1) \sigma_1(x)-\sigma_2(x)\right) \\ 
\notag 
     & \phantom{=\frac{1}{4}\ } + 
     \frac{1}{4} \left(x^2 \sigma_{-2,3}(x)+x(x-3) \sigma_{-1,3}(x) - 
     (3x-2) \sigma_{0,3}(x)+2\sigma_{1,3}(x)\right) \\ 
\notag 
\tau_{2}^{(1)}(k) & = 
     \frac{1}{4} \left((3-k) k \sigma_0(k)+2(k-1) \sigma_1(k)-2\sigma_2(k)\right) \\ 
\notag 
     & \phantom{=\frac{1}{4}\ } + 
     \frac{1}{4} \left(k^2 \sigma_{-1,3}(k)+(k-3) k \sigma_{0,3}(k) - 
     (3k-2) \sigma_{1,3}(k)+ 2\sigma_{2,3}(k)\right) \\ 
\notag 
\rho_{2}^{(1)}(x) & = 
\frac{1}{4}\left( \left(x^2-1\right) \sigma_1(x)-x(x-3) \sigma_0(x)- 
     2\sigma_2(x)\right) \\ 
\notag 
     & \phantom{=\frac{1}{4}\ } + 
     \frac{1}{4}\left(x^2 \sigma_{-1,3}(x) + x(x-3) \sigma_{0,3}(x) - 
     (3x-2) \sigma_{1,3}(x)+ 2\sigma_{2,3}(x)\right). 
\end{align} 
We simplify our resulting 
finite sum expansions using the next identities 
to find expansions that follow by taking the order $s := 2$ derivatives of 
$L_{\alpha}(q)$ 
(compare Table~\ref{table_bdd-divisor_divsumfns_listings} on 
page~\pageref{table_bdd-divisor_divsumfns_listings}): 
\begin{equation*} 
\sigma_{\alpha,1}(n) = \sigma_{\alpha}(n), 
     \qquad \text{ \ \textrm{and} \ } \qquad 
\sigma_{\alpha,2}(n) = \sigma_{\alpha}(n) - n^{\alpha},\ 
     \mathrm{for\ all\ } n \geq 1.  
\end{equation*} 
\end{example} 

For the higher-order derivative cases where $s, \alpha \geq 2$, 
a corresponding set of component functions (parameterized in $\alpha$) is defined 
to state the more general formulas compared to the prior 
summations where $\alpha \in \{0,1\}$ in the last example. 

\begin{example}[Generalization to symbolic parameters, $\alpha \in \mathbb{R}$] 
\label{example_Intro2} 
For integers $s \geq 2$, we will denote the analogous coefficients by 
$\tau_{s,x}^{(\alpha)}(k)$ and $\rho_s^{(\alpha)}(x)$. 
The exact polynomial expansions in $x$ and the parameters $(s,k)$ that result have 
growing complexity that should roughly be expressable by closed-form 
representations of classically Stirling-like polynomials. 
Here, we may write the generalized sum-of-divisors functions for any fixed parameter 
$\alpha \in \mathbb{C}$ in the form of 
\[
\binom{x}{s} \sigma_{\alpha}(x) = 
      \rho_s^{(\alpha)}(x) + \sum_{k=1}^{x-1} \tau_{s,x}^{(\alpha)}(k), 
\]
for natural numbers $x \geq 2$. 
Then we have the next expansions of the corresponding 
coefficient functions defined (and then summed exactly) as 
\begin{align} 
\notag 
\tau_{2,x}^{(\alpha)}(k) & = 
\frac{1}{4} \left(k^2 \sigma_{\alpha-2 ,3}(k)+k (k-3) \sigma_{\alpha-1 ,3}(k)- 
     (3k-2) \sigma_{\alpha ,3}(k)+2\sigma_{\alpha+1 ,3}(k)\right) \\ 
\notag 
     & \phantom{=\frac{1}{4}\ } - 
     \frac{1}{4} \left(k^2 \sigma_{\alpha-2 }(k) - 
     k(k-3) \sigma_{\alpha-1 }(k)+(3k-2) \sigma_{\alpha }(k) - 
     2\sigma_{\alpha+1 }(k)\right) \\ 
\notag 
\rho_{2}^{(\alpha)}(x) & = 
\frac{1}{4} \left(x^2 \sigma _{\alpha -2,3}(x)+ 
     x(x-3) \sigma _{\alpha -1,3}(x)-(3x-2) \sigma _{\alpha ,3}(x)+ 
     2 \sigma _{\alpha +1,3}(x)\right) \\ 
\notag 
     & \phantom{=\frac{1}{4}\ } - 
     \frac{1}{4} \left(x^2 \sigma _{\alpha -2}(x)- 
     \left(2 x^2+x-2\right) \sigma _{\alpha }(x)+x(x-3) \sigma _{\alpha -1}(x)\right) \\ 
\notag 
     & \phantom{=\frac{1}{4}\ } - 
     \frac{1}{2}\sigma_{\alpha +1}(x) \\ 
\notag 
\tau_{3,x}^{(\alpha)}(k) & = 
     -\frac{1}{6} k^3 \sigma _{\alpha -3,3}(k)+ 
     \frac{1}{18} k^3 \sigma _{\alpha -3,4}(k)- 
     \left(\frac{k^2}{6}-\frac{11 k}{12}+
     \frac{1}{3}\right) \sigma _{\alpha ,4}(k) \\ 
\notag 
     & \phantom{=\ } + 
     \left(\frac{3 k^2}{4}-\frac{1}{12} k (9 x+17)+ 
     \frac{x}{2}\right) \sigma _{\alpha ,3}(k)+ 
     \left(\frac{k^3}{12}-\frac{k^2}{3}\right) \sigma _{\alpha -2,4}(k) \\ 
\notag 
     & \phantom{=\ } + 
     \left(\frac{k^3}{36}-\frac{k^2}{2}+ 
     \frac{11 k}{18}\right) \sigma _{\alpha -1,4}(k)+ 
     \left(\frac{1}{4} k^2 (x+2)- 
     \frac{k^3}{3}\right) \sigma _{\alpha -2,3}(k) \\ 
\notag 
     & \phantom{=\ } + 
     \left(-\frac{k^3}{6}+ 
     \frac{1}{4} k^2 (x+5)- 
     \frac{1}{12} k (9 x+4)\right) \sigma _{\alpha -1,3}(k) \\ 
\notag 
     & \phantom{=\ } + 
     \left(\frac{11 k}{36}-\frac{1}{2}\right) \sigma _{\alpha +1,4}(k) \\ 
\notag 
     & \phantom{=\ } + 
     \frac{1}{2} \sigma _{\alpha +2,3}(k)- 
     \frac{1}{6} \sigma _{\alpha +2,4}(k)+ 
     \left(\frac{x+1}{2}-\frac{13 k}{12}\right) \sigma _{\alpha +1,3}(k) \\ 
\notag 
     & \phantom{=\ } + 
     \frac{1}{9} k^3 \sigma _{\alpha -3}(k) - 
     \left(\frac{7 k^2}{12}-\frac{1}{4} k (3 x+2)+ 
     \frac{1}{6} (3 x-2)\right) \sigma _{\alpha }(k) \\
\notag 
     & \phantom{=\ } + 
     \left(\frac{k^3}{4}- 
     \frac{1}{12} k^2 (3 x+2)\right) \sigma _{\alpha -2}(k) \\ 
\notag 
     & \phantom{=\ } + 
     \left(\frac{5 k^3}{36}-\frac{1}{4} k^2 (x+3)+ 
     \frac{1}{36} k (27 x-10)\right) \sigma _{\alpha -1}(k)- 
     \frac{1}{3} \sigma _{\alpha +2}(k) \\ 
\notag 
     & \phantom{=\ } + 
     \left(\frac{7 k}{9}- 
     \frac{x}{2}\right) \sigma _{\alpha +1}(k) \\ 
\notag 
\rho_{3}^{(\alpha)}(x) & = 
     -\frac{1}{6} x^3 \sigma _{\alpha -3,3}(x)+ 
     \frac{1}{18} x^3 \sigma _{\alpha -3,4}(x)- 
     \frac{1}{36} (x-18) x^2 \sigma _{\alpha -2,3}(x) \\ 
\notag 
     & \phantom{=\ } + 
     \frac{1}{12} (x-4) x^2 \sigma _{\alpha -2,4}(x) \\ 
\notag 
     & \phantom{=\ } + 
     \frac{1}{6} \left(x^2+2 x-2\right) x \sigma _{\alpha -1,3}(x)+ 
     \frac{1}{36} \left(x^2-18 x+22\right) x \sigma _{\alpha -1,4}(x) \\ 
\notag 
     & \phantom{=\ } + 
     \frac{1}{36} \left(x^2-9 x-29\right) x \sigma _{\alpha ,3}(x)- 
     \frac{1}{12} \left(2 x^2-11 x+4\right) \sigma _{\alpha ,4}(x) \\ 
\notag 
     & \phantom{=\ } - 
     \frac{1}{12} (x-1) (x+6) \sigma _{\alpha +1,3}(x)+ 
     \frac{1}{36} (11 x-18) \sigma _{\alpha +1,4}(x) \\ 
\notag 
     & \phantom{=\ } + 
     \frac{1}{18} (x+9) \sigma _{\alpha +2,3}(x) - 
     \frac{1}{6} \sigma _{\alpha +2,4}(x)+ 
     \frac{1}{9} x^3 \sigma _{\alpha -3}(x) \\ 
\notag 
     & \phantom{=\ } - 
     \frac{1}{18} (x+3) x^2 \sigma _{\alpha -2}(x) - 
     \frac{1}{36} \left(7 x^2-6 x+10\right) x \sigma _{\alpha -1}(x) \\ 
\notag 
     & \phantom{=\ } + 
     \frac{1}{36} \left(5 x^3-3 x^2+8 x+12\right) \sigma _{\alpha }(x)+ 
     \frac{1}{36} (3 x+4) x \sigma _{\alpha +1}(x) \\ 
\label{eqn_TauRho_Idents_for_s23Alpha} 
     & \phantom{=\ } - 
     \frac{1}{18} (x+6) \sigma _{\alpha +2}(x). 
\end{align} 
\end{example} 

The most general statements that express the behavior of the expansions for 
$\sigma_{\alpha}(n)$ of this type are given in 
Theorem~\ref{theorem_FiniteSum_MainInitialThmStmt} and 
Corollary~\ref{cor_LHSSigmaAlphan_Exact_Formula_exprs}. 
The reductions to coefficients as 
finite-degree polynomials of $x$ demonstrated by 
Example~\ref{example_dn_sigman_formula_exps_s2} and 
Example~\ref{example_Intro2} 
serve as a key indication 
for why these complicated coefficient expressions are still 
worthwhile and significant to study as new algebraic identities in relation to the classical 
functions, $\sigma_{\alpha}(n)$. 

\section{Construction of the main results} 
\label{Section_DefsStmtsOfMainResults} 

\subsection{Definitions} 
\label{subSection_Defs} 

\begin{definition} 
\label{def_Multiple_Init_Coeff_Sums} 
For integers $s \geq 1$ and $r,p,m,u,n,w \geq 0$, we define the 
following single and multiple coefficient sums expanded by 
\begin{align*} 
C_{4,s}^{(\alpha)}(r, p, m) & := 
     \sum_{k=0}^m \binom{s}{r} \gkpSI{s-r}{m} \gkpSII{m}{k} \binom{s-r-k}{p} 
     (-1)^{s-r-k+p} k! \cdot r! \\ 
C_{8,s}^{(\alpha)}(r, p, n, w) & := 
     \sum_{m=0}^{s-r} \sum_{k=0}^m 
     \sum_{m_1=0}^{r+1-n} \sum_{m_2=0}^n 
     \binom{s}{r} \gkpSI{s-r}{m} \gkpSII{m}{k} \binom{s-r-k}{p} \gkpSI{r+1-n}{m_1} \times \\
     & \phantom{:=\sum\ } \times 
     \gkpSI{n}{m_2} 
     \binom{m_2}{w-m-m_1} 
     (-1)^{s-k+p+n+m_1+m_2} k! \times \\ 
     & \phantom{:=\sum\ } \times 
     (n-2-r)^{m+m_1+m_2-w} \\ 
C_{42,s}^{(\alpha)}(x, k; p, m, u) & := 
     \sum_{r=0}^s \sum_{j=0}^{s-r} \binom{x-k}{r} \gkpSI{s-r}{j} \binom{j}{u} 
     \frac{(-1)^{s-r-u} (p+1-s+r)^{j-u} k^u}{(s-r)!} \times \\ 
     & \phantom{:=\sum\ } \times 
     C_{4,s}^{(\alpha)}(r, p, m) \\ 
C_{43,s}^{(\alpha)}(x, k; p, m, u) & := 
     \sum_{r=0}^s \sum_{j=0}^{s-r} \binom{x+1-k}{r+1} \gkpSI{s-r}{j} \binom{j}{u} \times \\ 
     & \phantom{:=\sum\ } \times 
     \frac{(-1)^{s-r-u} (r+1) (p+1-s+r)^{j-u} k^u}{(s-r)!} 
     C_{4,s}^{(\alpha)}(r, p, m) \\ 
C_{82,s}^{(\alpha)}(x, k; w, p, u) & := 
     \sum_{r=0}^s \sum_{n=0}^{r+1} \sum_{j=0}^{s-r} 
     \binom{x+1-n-k+r}{r} \binom{r}{n} \gkpSI{s-r}{j} \binom{j}{u} \times \\ 
     & \phantom{:=\sum\sum\sum\ } \times 
     \frac{(-1)^{s-r-u} (p+2-s+r)^{j-u} k^u}{(s-r)!} 
     C_{8,s}^{(\alpha)}(r, p, n, w) \\ 
C_{83,s}^{(\alpha)}(x, k; w, p, u) & := 
     \sum_{r=0}^s \sum_{n=0}^{r+1} \sum_{j=0}^{s-r} 
     \binom{x+2-n-k+r}{r+1} \binom{r+1}{n} \gkpSI{s-r}{j} \binom{j}{u} \times \\ 
     & \phantom{:=\sum\sum\sum\ } \times 
     \frac{(-1)^{s-r-u} (p+2-s+r)^{j-u} k^u}{(s-r)!} 
     C_{8,s}^{(\alpha)}(r, p, n, w). 
\end{align*} 
\end{definition} 

\begin{definition} 
\label{def_Sum4sqi_Sum8sqi_init_defs} 
For fixed integers $s, i \geq 1$ and an indeterminate series parameter $|q| < 1$, 
we define the next two primary component sums as follows: 
\begin{align*} 
\Sum_{4,s}(q, i) & := \sum_{0 \leq r, p, m \leq s} 
     \frac{q^{(p+1)i+r}}{(1-q^i)^{s-r+1}} \left( 
     \frac{i^{m+1}}{(1-q)^{r+1}} - \frac{(r+1) i^m}{(1-q)^{r+2}} \right) 
     C_{4,s}(r, p, m), \\ 
\Sum_{8,s}(q, i) & := \sum_{0 \leq r, p \leq s} \sum_{n=0}^{r+1} \sum_{w=0}^s 
     \frac{q^{(p+1)i+n-1}}{(1-q^i)^{s-r+1}} \times \\ 
     & \phantom{:=\sum\sum\sum\sum\ } \times 
     \left( 
     \binom{r}{n} \frac{1}{(1-q)^{r+1}} - \binom{r+1}{n} \frac{1}{(1-q)^{r+2}} 
     \right) \times \\ 
     & \phantom{:=\sum\sum\sum\sum\ } \times 
     i^w C_{8,s}(r, p, n, w). 
\end{align*} 
\end{definition} 

\begin{definition} 
\label{def_SeriesCoeffs_ShorthandNotation} 
For fixed real $\alpha \geq 0$ and integers $s, x \geq 1$ we define the following 
coefficients of several intermediate power series expansions of the 
Lambert series expansions from the introduction: 
\begin{align*} 
\tag{i} 
L_{s,x}^{(\alpha)} & := 
     [q^x] q^s D^{(s)}\left[\frac{L_{\alpha}(q)}{1-q}\right] \\ 
\tag{ii} 
S_{00,s,x}^{(\alpha)} & := 
     [q^x] q^s D^{(s)}\left[\sum_{i \geq 1} \frac{i^{\alpha-1} q^i}{(1-q)^2} 
     \right] \\ 
\tag{iii} 
S_{01,s,x}^{(\alpha)} & := 
     [q^x] q^s D^{(s)}\left[\sum_{i \geq 1} \frac{i^{\alpha-1} q^i}{1-q} 
     \right] \\ 
S_{4,s,x}^{(\alpha)} & := [q^x] \sum_{i \geq 1} \Sum_{4,s}(q, i) i^{\alpha-1} \\ 
S_{8,s,x}^{(\alpha)} & := [q^x] \sum_{i \geq 1} \Sum_{8,s}(q, i) i^{\alpha-1}. 
\end{align*} 
In the corollaries of the new results given in later subsections 
of the article, we also employ the shorthand notation of 
$S_{48,s,x}^{(\alpha)} := S_{4,s,x}^{(\alpha)} + S_{8,s,x}^{(\alpha)}$.  
\end{definition} 

By Lemma~\ref{lemma_formulas_for_special_series_coeffs}, 
we have the following explicit formulas for the first three tagged 
coefficient functions from 
Definition~\ref{def_SeriesCoeffs_ShorthandNotation}: 
\begin{align*} 
\tag{i} 
L_{s,x}^{(\alpha)} & = 
     \sum_{r=0}^s \sum_{k=1}^x 
     \binom{s}{r} \binom{x-k}{s-r} \frac{(s-r)! k!}{(k-r)!} \sigma_{\alpha}(k) \\ 
\tag{ii} 
S_{00,s,x}^{(\alpha)} & = 
     \sum_{r=0}^s \sum_{k=0}^x 
     \binom{s}{r} \binom{x-k+1}{s-r+1} \frac{(s-r+1)! k^{\alpha-1} \cdot k!}{ 
     (k-r)!} \\ 
\tag{iii} 
S_{01,s,x}^{(\alpha)} & = 
     \sum_{r=0}^s \sum_{k=0}^x 
     \binom{s}{r} \binom{x-k}{s-r} \frac{(s-r)! k^{\alpha-1} \cdot k!}{ 
     (k-r)!}. 
\end{align*} 

\subsection{Statements of the new formulas} 
\label{subSection_prop_KeyProp_sthDerivsOfLSeriesExps}

\begin{prop}[Higher-order derivatives of LGF expansions] 
\label{prop_KeyProp_sthDerivsOfLSeriesExps} 
For a fixed indeterminate series parameter $|q| < 1$, 
any real-valued $\alpha \geq 0$, and any integers 
$s \geq 1$, we have that 
\begin{align*} 
q^s D^{(s)}\left[\frac{1}{1-q} \times \sum_{i \geq 1} 
     \frac{i^{\alpha} q^i}{1-q^i}\right] & = 
     q^s D^{(s)}\left[\sum_{i \geq 1} 
     \frac{i^{\alpha} q^i}{(1-q)^2}\right] \\ 
     & \phantom{=\quad\ } + 
     \sum_{i \geq 1} \left(\Sum_{4,s}(q, i) + 
     \Sum_{8,s}(q, i)\right) i^{\alpha-1}. 
\end{align*} 
Hence, for all integers $s, x \geq 1$, we have that 
\begin{align} 
\label{eqn_prop_KeyProp_sthDerivsOfLSeriesExps_v2} 
L_{s,x}^{(\alpha)} & = S_{00,s,x}^{(\alpha)} + S_{4,s,x}^{(\alpha)} + 
     S_{8,s,x}^{(\alpha)}, 
\end{align} 
\end{prop} 

The proof of Proposition~\ref{prop_KeyProp_sthDerivsOfLSeriesExps} is 
somewhat involved and requires the machinery of several lemmas we will prove 
in the next section. For this reason we delay the proof of this key result 
until Section~\ref{subSection_ProofOfKeyLSProp}. 
The consequence stated in \eqref{eqn_prop_KeyProp_sthDerivsOfLSeriesExps_v2} 
follows trivially from the first result in the proposition. 
In particular, we immediately arrive at the second formula 
since the left-hand-side of the series in the leading equation is equal to $q^s$ times the 
$s^{th}$ derivative of the generating function, $L_{\alpha}(q) (1-q)^{-1}$. 

\begin{cor} 
\label{cor_S4sx_S8sx_exps} 
For any integer $x \geq 1$, we have the following expansions: 
\begin{align*} 
\tag{i} 
S_{4,s,x}^{(\alpha)} & = \sum_{0 \leq r, p, m \leq s} \sum_{k=1}^{x-r} \Biggl( 
     \binom{x-k}{r} B_{s-r,p+1}(m+\alpha; k) \\ 
     & \phantom{=\sum\sum\sum\Biggl(\quad\ } - 
     \binom{x+1-k}{r+1}(r+1) B_{s-r,p+1}(m+\alpha-1; k)\Biggr) 
     C_{4,s}^{(\alpha)}(r, p, m); \\ 
S_{8,s,x}^{(\alpha)} & = \sum_{0 \leq r, p, w \leq s} \sum_{n=0}^{r+1} 
     \sum_{k=1}^{x+1-n} \Biggl(\binom{x+1-n-k+r}{r} \binom{r}{n} \\ 
     & \phantom{\sum_{0 \leq r, p, w \leq s} \sum_{n=0}^{r+1} 
     \sum_{k=1}^{x+1-n} \Biggl(\quad\ } - 
     \binom{x+2-n-k+r}{r+1} \binom{r+1}{n}\Biggr) \times \\ 
\tag{ii} 
     & \phantom{=\sum\sum\sum\sum\Biggl(\qquad\ } \times 
     C_{8,s}^{(\alpha)}(r, p, n, w) B_{s-r,p+2}(w+\alpha-1; k). 
\end{align*} 
\end{cor} 

We prove only the second result stated in 
(ii) of Corollary \ref{cor_S4sx_S8sx_exps} and 
leave the details of the full argument to the proof of (i) as a 
related exercise. 
Since we employ the 
result in \eqref{eqn_BinCvlDivFn_as_CoeffsOfLSeries_exp} 
to justify these expansions in our argument, 
we first sketch a proof of this identity. 

\begin{proof}[Proof of \eqref{eqn_BinCvlDivFn_as_CoeffsOfLSeries_exp}]  
For fixed integers $k \geq 0$ and 
$m,i \geq 1$, consider the expansion of the following terms through the 
known binomial series identity where we select $|q| < 1$ by assumption 
\cite[\S 5.4]{GKP}: 
\begin{align*} 
L_{i,k,m}(q) & := \frac{f(i) q^{mi}}{(1-q^i)^{k+1}} = 
     f(i) \times \sum_{j=0}^{\infty} \binom{j+k}{k} q^{(m+j)i}. 
\end{align*} 
Since we sum over all $i \geq 1$, the coefficients of $q^{n}$ in the 
full power series expansion of $L_{i,k,m}(q)$ in $q$ about zero 
must satisfy the integer relation that 
$(m+j) i = n$. 
This equation naturally leads to expressing the coefficients of $q^n$ as 
a sum over the divisors $i$ of $n$ with summands weighted by the coefficients 
expanded above. 

To conclude the proof of the identity in 
\eqref{eqn_BinCvlDivFn_as_CoeffsOfLSeries_exp}, 
we must then (I) solve for the input $j = \frac{n}{i}-m$ depending on the other 
fixed integer parameters which allows us to express the inputs to the 
binomial coefficient terms in the 
resulting coefficient divisor sums; and then 
(II) notice that the divisors $i$ of $n$ are positive integers 
bounded by $0 < i = \frac{n}{m+j} \leq \frac{n}{m}$. 
These two steps imply that     
\begin{align*} 
[q^n] \sum_{d \geq 1} L_{d,k,m}(q) & = 
     \sum_{\substack{d|n \\ d \leq \lfloor \frac{n}{m} \rfloor}} 
     \binom{\frac{n}{d}-m+k}{k} f(d). 
\end{align*} 
The identity we cited in 
\eqref{eqn_BinCvlDivFn_as_CoeffsOfLSeries_exp} of the introduction follows 
as the special case where we set $f(n) := n^{\alpha}$ for some real-valued 
parameter $\alpha \geq 0$. 
\end{proof} 
\begin{proof}[Proof of (ii) in Corollary~\ref{cor_S4sx_S8sx_exps}] 
We apply \eqref{eqn_BinCvlDivFn_as_CoeffsOfLSeries_exp} 
to our function, $\Sum_{8,s}(q, i)$, 
defined in the last subsection along with the known Cauchy product 
formula (finite summation formula) for the coefficients of the 
convolution of two ordinary generating functions in $q$ to obtain the 
claimed result. 
More precisely, we see that for $r_0 := 0, 1$ we have a difference of 
coefficients of the form 
\begin{align*} 
\sum_{r,p,n,w \geq 0} & \sum_{i \geq 1} \binom{r+r_0}{n} 
     \frac{i^{w+\alpha-1} q^{(p+1)i+n-1}}{ 
     (1-q^i)^{s-r+1} \times (1-q)^{r+1+r_0}} \\ 
     & = 
     \sum_{\substack{r,p,n,w \geq 0 \\ i \geq 1, x \geq 0}} 
     \sum_{k=0}^x \binom{r+r_0}{n} \binom{x-k+r+r_0}{r+r_0} 
     B_{s-r,p+1}(w+\alpha-1; k) \cdot q^{x+n-1}. 
\end{align*} 
If $G(q)$ denotes the ordinary generating function of the sequence 
$\{g_n\}_{n \geq 0}$, then for any integers $x > n \geq 0$ we have 
by standard operations on power series that 
$[q^x] q^{n-1} G(q) = g_{x+1-n}$. 
So we can shift the indices of the coefficients of $q^x$ by $n-1$ 
in the previous equation to obtain the complete proof of the 
second expansion stated in (ii) of the corollary above. 
\end{proof} 

Under the assumption of 
Proposition~\ref{prop_KeyProp_sthDerivsOfLSeriesExps}, 
we are able to state and prove precise expansions of 
$L_{s,x}^{(\alpha)}$ via the formula in 
\eqref{eqn_prop_KeyProp_sthDerivsOfLSeriesExps_v2}. 
What results is the next theorem that 
generalizes the form of the two examples we gave in the introduction 
for any fixed integers $s \geq 0$ and real parameters $\alpha \geq 0$.  

\begin{theorem}[Expansions by the bounded divisor sum functions] 
\label{theorem_FiniteSum_MainInitialThmStmt} 
Suppose that $\alpha \geq 0$ and $s \geq 0$ is a fixed integer. 
For all $x \geq 1$, we have the following expansions: 
\begin{align*} 
S_{4,s,x}^{(\alpha)} & = \sum_{0 \leq p,m,u \leq s} \sum_{k=1}^x \left( 
     C_{42,s}^{(\alpha)}(x, k; p, m, u) - C_{43,s}^{(\alpha)}(x, k; p, m+1, u)\right) 
     \sigma_{m+\alpha-u,p+1}(k) \\ 
     & \phantom{=\ } - 
     \sum_{0 \leq p, u \leq s} \sum_{k=1}^x C_{43,s}^{(\alpha)}(x, k; p, 0, u) 
     \sigma_{\alpha-1-u,p+1}(k) \\ 
S_{8,s,x}^{(\alpha)} & = \sum_{0 \leq p,u,w \leq s} \sum_{k=1}^{x+1} 
     \left(C_{82,s}^{(\alpha)}(x, k; w, p, u) - C_{83,s}^{(\alpha)}(x, k; w, p, u) 
     \right) \sigma_{w+\alpha-1-u,p+2}(k). 
\end{align*} 
\end{theorem} 

\begin{proof} 
We can expand the forms of the binomial coefficients 
that are coefficients of the summands of the divisor sums that define 
the functions, $B_{k,m}(\alpha; n)$, in the 
previous corollary. 
For integers $j \geq 0$, the binomial coefficient $\binom{r}{j}$ 
has a degree-$j$ polynomial expansion in $r$ weighted 
by the Stirling numbers of the first kind that naturally arise in expanding 
the rising and falling factorial functions. 
This interpretation allows us to write the inner sum terms defining the functions in 
\eqref{eqn_modified_divisorFns_with_bdd_divs} as polynomials in $n / d$. 

Namely, suppose that $t$ is an indeterminate and that the non-negative integers 
$m, k$ are fixed. Then we readily see the following identity: 
\begin{align*} 
\binom{t-m+k}{k} & = \sum_{j=0}^k \gkpSI{k}{j} \frac{(-1)^{k-j}}{k!} (t-m+k)^j \\ 
     & = 
     \sum_{r=0}^k \left(\sum_{j=0}^k \gkpSI{k}{j} \binom{j}{r} 
     \frac{(-1)^{k-r} (m-k)^{j-r}}{k!}\right) t^r. 
\end{align*} 
The rest of the proof of the theorem involved 
rearranging terms in the first expansions from 
Corollary~\ref{cor_S4sx_S8sx_exps} 
and shifting the indices $m$ from the previous equation 
upward by one to adjust the resulting formula. 
\end{proof} 

\begin{remark}[Simplifications of the multiple summation formulas]
It is well known that the Stirling number triangles satisfy inversion 
(orthogonality) relations 
of the following forms for any integers $m,n \geq 0$ \cite[\S 6.1]{GKP}: 
\begin{align*} 
\sum_{k=0}^{n} \gkpSI{n}{k} \gkpSII{k}{m} (-1)^{n-k} & = \delta_{m,n} \\ 
\sum_{k=0}^{n} \gkpSII{n}{k} \gkpSI{k}{m} (-1)^{n-k} & = \delta_{m,n}. 
\end{align*}
These properties combined with established bivariate triangular recurrence relations 
that characterize each of these triangles \cite[\cf \S 6.1]{GKP} 
suggest that in general there is much cancellation to be expected 
in the multiple summation formulas we defined above. 

An examination of the examples from 
the introduction confirms this intuition for a few small special cases of the theorem above. 
We take away from this observation that the formulas for $\sigma_{\alpha}(x)$ 
defined in terms of the last few quasi-polynomially scaled finite sums over the bounded-index 
variants of these functions convey new, significant substructure to these functions. 
The next section suggests interpretations of the formulas we have proved within this 
subsection along the lines of the characteristic expansions we noted result in the examples 
outlined as applications in the introduction to this article. 
\end{remark} 

\subsection{Corollaries and applications of the new identities} 

\begin{cor}[Exact Formulas for the generalized sums of divisors functions] 
\label{cor_LHSSigmaAlphan_Exact_Formula_exprs} 
For real-valued $\alpha \geq 0$, any integer $s \geq 2$, and all 
$x \geq 1$, we have that 
\begin{align*} 
\binom{x}{s} \sigma_{\alpha}(x) & = \frac{1}{s!} \left( 
     S_{01,s,x}^{(\alpha)} + S_{48,s,x}^{(\alpha)} - 
     S_{48,s,x-1}^{(\alpha)} + s \cdot S_{48,s-1,x-1}^{(\alpha)} 
     \right). 
\end{align*} 
\end{cor} 
\begin{proof}  
Fix some natural number $j \geq 1$. 
We see that for any function, $f(q)$, that is taken to be $k$-order differentiable 
for all $1 \leq k \leq j$, we obtain the following 
expansion of the higher order derivatives of $f$ by induction: 
\begin{align} 
\label{eqn_Djqm1Timesfq_DerivIdent_stmt} 
D^{(j)}\left[(1-q) f(q)\right] & = 
     (1-q) f^{(j)}(q) - j f^{(j-1)}(q),\ j \geq 1. 
\end{align} 
Theorem~\ref{theorem_FiniteSum_MainInitialThmStmt} is 
proved starting from the results in 
Proposition~\ref{prop_KeyProp_sthDerivsOfLSeriesExps} from the previous subsection. 
The idea in stating the proposition is to scale the Lambert series 
generating function, $L_{\alpha}(q)$, by a multiple (factor) of $(1-q)^{-1}$ prior to 
differentiating the series. It happens that the resulting formulas for the 
$s^{th}$-order derivatives of the series after the rescaling are less 
complicated in form to state than those corresponding to the original Lambert series 
generating function for $\sigma_{\alpha}(x)$. 
If we then revisit our identities for the $s^{th}$ derivatives of 
the formula in \eqref{eqn_key_LSeries_exp_ident} and subsequently re-interpret our results 
using the expansion we observe in \eqref{eqn_Djqm1Timesfq_DerivIdent_stmt}, we find 
stated exact formulas for $\sigma_{\alpha}(n)$ when $j = s$. 
\end{proof} 

For fixed $x, k \in \mathbb{Z}^{+}$, we can define component functions in our 
results by setting 
\begin{align} 
\label{eqn_example_SxsxkAlpha_ModCfDefs} 
\sum_{k=1}^{x-1} S_{01,s,x,k}^{(\alpha)} & \quad \leftmapsto: \quad S_{01,s,x}^{(\alpha)} \\ 
\notag 
\sum_{k=1}^{x-1} S_{4,s,x,k}^{(\alpha)} & \quad \leftmapsto: \quad S_{4,s,x}^{(\alpha)} \\ 
\notag 
\sum_{k=1}^{x-1} S_{8,s,x,k}^{(\alpha)} & \quad \leftmapsto: \quad S_{8,s,x}^{(\alpha)}, 
\end{align} 
where we write the shorthand 
$$S_{48,s,x,k}^{(\alpha)} := S_{4,s,x,k}^{(\alpha)} + S_{8,s,x,k}^{(\alpha)}.$$ 
The motivation for the definition in terms of $k,x$ given in 
\eqref{eqn_example_SxsxkAlpha_ModCfDefs} is to provide a mechanism for 
simplifying the complicated nested summations 
exactly expressing the characteristic leading coefficients in the new formulas. 

We then seek exact closed-form expressions for $\sigma_{\alpha}(x)$ 
involving sums over the bounded-index functions 
$\sigma_{\alpha-j,m}(k)$ for integers $j$ and $m \geq 1$. 
In particular, we may use the intermediate definitions of these 
sums using the shorthand notation 
from \eqref{eqn_example_SxsxkAlpha_ModCfDefs} to express the 
result from Corollary~\ref{cor_LHSSigmaAlphan_Exact_Formula_exprs} as\footnote{ 
     For positive integers $\beta \geq 0$, we have formulas for the 
     polynomial power sum functions expanded by the \emph{Bernoulli numbers} and 
     \emph{Bernoulli polynomials} given by \cite[\S 24.4(iii)]{NISTHB} 
     \begin{align*} 
     \sum_{k=0}^{n} k^{\beta} & = \frac{B_{\beta+1}(n+1) - B_{\beta+1}}{\beta+1}. 
     \end{align*} 
} 
\begin{align} 
\notag 
\binom{x}{s} \sigma_{\alpha}(x) & \phantom{:} = 
     S_{48,s,x,x}^{(\alpha)} + 
     \sum_{k=1}^{x-1} \frac{1}{s!} \left( 
     S_{01,s,x,k}^{(\alpha)} + S_{48,s,x,k}^{(\alpha)} - 
     S_{48,s,x-1,k}^{(\alpha)} + S_{48,s-1,x-1,k}^{(\alpha)}\right)  \\ 
\label{eqn_SigmaAlphan_RHS_SumExpr_SumTauPlusRho} 
     & := 
     \rho_s^{(\alpha)}(x) + \sum_{k=1}^{x-1} \tau_{s,x}^{(\alpha)}(k). 
\end{align} 
In general, we can prove formally by induction that for functions, 
$p_{i,s,\beta}(k, x)$ and $q_{i,s,\beta}(x)$ polynomial in 
their respective inputs, 
we have expansions of the component functions to the formulas in 
\eqref{eqn_SigmaAlphan_RHS_SumExpr_SumTauPlusRho} that satisfy the following forms: 
\begin{align} 
\label{eqn_TauRhosxkx_GenFormPolyCoeffExps} 
\tau_{s,x}^{(\alpha)}(k) & = \sum_{\beta=-s}^{s-1} \left[
    p_{1,s,\beta}(k, x) \cdot \sigma_{\alpha+\beta}(k) + 
    \sum_{m=1}^{s+1} 
    p_{2,s,\beta}(k, x) \cdot \sigma_{\alpha+\beta,m}(k) 
    \right] \\ 
\notag 
\rho_{s}^{(\alpha)}(x) & = \sum_{\beta=-s}^{s-1} \left[ 
     q_{1,s,\beta}(x) \cdot \sigma_{\alpha+\beta}(x) + 
     \sum_{m=1}^{s+1} 
     q_{2,s,\beta}(x) \cdot \sigma_{\alpha+\beta,m}(x) 
     \right]. 
\end{align} 

\section{Complete proofs of the new results} 
\label{Section_appendix_proofs_of_lemmas} 

\subsection{Proofs of key lemmas from the introduction} 

\begin{lemma}
\label{lemma_LambertSeriesTerm_Sumjim2_Exps} 
For any indeterminate $q$ and integers $i \geq 2$, we have the following 
expansion: 
\begin{align} 
\label{eqn_lemma_LambertSeriesTerm_Sumjim2_Exps_stmt_v1}
\frac{1}{1-q^i} & = \frac{1}{i}\left(\frac{1}{1-q}+\sum_{j=0}^{i-2} 
     \frac{(i-1-j) q^j (1-q)}{1-q^i}\right). 
\end{align} 
\end{lemma} 
\begin{proof} 
We start by noticing that $(1-q^i) / (1-q) = 1+q+\cdots+q^{i-1}$ by a finite 
geometric series expansion. If we combine denominators of the two 
fractions on the right-hand-side of 
\eqref{lemma_LambertSeriesTerm_Sumjim2_Exps} 
we obtain that 
\begin{align*} 
\frac{1+q+\cdots+q^{i-1} + (1-q) \sum_{j=0}^{i-2} (i-1-j) q^j}{1-q^i}. 
\end{align*} 
Then by simplifying the summation terms in the numerator of the above 
expansion we see that 
\begin{align*} 
(1-q) \sum_{j=0}^{i-2} (i-1-j) q^j & = i-1 - q^{i-1} + 
     \sum_{j=0}^{i-2} \left((i-1-j) - (i-2-j)\right) q^{j+1} \\ 
     & = 
     i - \sum_{j=0}^{i-1} q^j = i - \frac{1-q^i}{1-q}. 
\end{align*} 
Finally, dividing through by $1-q^i$ and taking a difference of terms proves the result. 
\end{proof} 

Lemma~\ref{lemma_sthDerivsOfLambertSeries} 
stated in the introduction provides another Lambert series transformation 
identity which we will need to prove to complete the argument in the 
proof of the proposition given in the next subsection. 

\begin{proof}[Proof of Lemma~\ref{lemma_sthDerivsOfLambertSeries}] 
\label{lemma_sthDerivsOfLambertSeries-ProofRef} 
We prove the first identity stated in (i) of the lemma by induction on $s$. 
When $s := 0$, the Stirling number terms are identically equal to one, and so 
we obtain $q^i \cdot (1-q^i)^{-1}$ on both sides of the equation. Next, we suppose that 
(i) is correct for all $t < s$ for some $s \geq 1$. 
That is, if we let 
\[
\Delta_{t,i}(q) := q^t \cdot D^{(t)}\left[\frac{q^i}{1-q^i}\right], 
\]
then we have that our inductive hypothesis is true in the form of 
\[
\tag{IH}
\Delta_{t,i}(q) = \sum_{m=0}^{t} \sum_{k=0}^{m} \gkpSI{t}{m} \gkpSII{m}{k} 
     \frac{(-1)^{t-k} k! \cdot i^m \cdot q^i}{(1-q^i)^{k+1}}, \forall 0 \leq t < s. 
\]
We have two triangular recurrence relations satisfied by each 
triangle of the first and second kinds. These recurrences are stated respectively 
as follows for integers $0 \leq k \leq n$ \cite[\S 6.1]{GKP}:
\begin{align*} 
\gkpSI{n}{k} & = (n-1) \gkpSI{n-1}{k} + \gkpSI{n-1}{k-1} + \delta_{n,0} \delta_{k,0} \\ 
\gkpSII{n}{k} & = k \gkpSII{n-1}{k} + \gkpSII{n-1}{k-1} + \delta_{n,0} \delta_{k,0}. 
\end{align*} 
We also notice by direct calculation that for integers $i \geq 1$ and $s \geq 0$, we have 
\begin{align*} 
q^{s+1} \cdot D\left[\frac{1}{q^s \cdot (1-q^i)^{k+1}}\right] & = 
     \frac{i(k+1) q^i}{(1-q^i)^{k+2}} - \frac{s}{(1-q^i)^{k+1}}. 
\end{align*} 
We have that 
\[
\frac{1}{q^s} \cdot \frac{q^i}{(1-q^i)^{k+1}} = 
     q^{-s}\left(\frac{1}{(1-q^i)^{k+1}} - \frac{1}{(1-q^i)^k}\right). 
\]
Let the shorthand $t_{m,k} \equiv t_{m,k}(i, q)$ be defined as 
\[
t_{m,k} := \frac{k \cdot q^i \cdot i^{m}}{(1-q^i)^{k+1}}. 
\]
Applying the triangular recurrence relation for the Stirling numbers of the first kind 
to the expansion of the (IH) by the identities in the last two equations leads to 
cancellation in the following form:
\begin{align*} 
\Delta_{s+1,i}(q) & = 
     q^{s+1} \cdot \frac{d}{dq}\left[\frac{1}{q^s} \cdot \Delta_{s,i}(q)\right] \\ 
     & = q^{s+1} \cdot \frac{d}{dq}\left[\frac{1}{q^{s-i}} \times \sum_{m=0}^{s} \sum_{k=0}^{m} 
     \gkpSI{s}{m} \gkpSII{m}{k} \frac{(-1)^{s-k} k! i^m}{(1-q^i)^{k+1}}\right] \\ 
     & = \sum_{m=0}^{s} \sum_{k=0}^{m} \gkpSI{s+1}{m+1} \gkpSII{m}{k} 
     \bigl(
     t_{m+1,k+1} - t_{m+1,k} 
     \bigr) (-1)^{s+1-k} k!. 
\end{align*} 
Then by applying summation by parts, the recurrence relation for the 
Stirling numbers of the second kind, and shifting the index of summation 
over $k$, we obtain that 
\begin{align*} 
\Delta_{s+1,i}(q) & = \sum_{m=0}^{s} \sum_{k=0}^{m} \gkpSI{s+1}{m+1} \gkpSII{m+1}{k+1} \times 
    t_{m+1,k+1} (-1)^{s+2-k} \cdot k! \\ 
    & \sum_{m=1}^{s+1} \sum_{k=1}^{m} \gkpSI{s+1}{m} \gkpSII{m}{k} \times 
    t_{m,k} (-1)^{s+1-k} \cdot (k-1)!. 
\end{align*} 
The identity in (i) finally follows by summing over $0 \leq k \leq m \leq s+1$ where 
$\gkpSI{s+1}{0} \equiv 0$ for all $s \geq 0$ \citep[\cf \S 26.8]{NISTHB}. 
The second identity stated in (ii) of the lemma is a consequence of (i) through 
an application of the binomial theorem in combination with the 
following identity:
\[
\frac{1}{(1-q^i)^{k+1}} = \frac{1}{1-q^i}\left[1+\frac{1}{1-q^i}\right]^{k}. 
\]
The complete form of the second identity follows by interchanging the order of summation. 
\end{proof} 

Lemma~\ref{lemma_SumOverjPowqj_SumIdents_v2} 
is used as another transformation of the higher-order 
derivatives of the Lambert series identity in 
\eqref{eqn_key_LSeries_exp_ident} that we have proved in 
Lemma~\ref{lemma_LambertSeriesTerm_Sumjim2_Exps} above. 
We prove (i) from this lemma in complete detail and then leave the 
similar proof of part (ii) as an exercise that follows 
easily by adapting the first argument. 

\begin{proof}[Proof of Equation (i) in Lemma~\ref{lemma_SumOverjPowqj_SumIdents_v2}] 
\label{lemma_SumOverjPowqj_SumIdents_v2-ProofRef} 
We first define the two sums, $S_{i,p,n}$, for $i = 1,2$ and positive integers 
$p, n \geq 1$ by 
\begin{align*} 
    S_{1,p,n} & := \sum_{k=0}^p \binom{p}{k} \frac{(-1)^{k+1} q^{n+1+k} (n+1)!}{ 
     (n-p)! (n+1-p+k) (1-q)^{p+1}} \\ 
S_{2,p,n} & := \sum_{j=0}^n \frac{j!}{(j-p)!} q^j - \frac{p! q^p}{(1-q)^{p+1}}. 
\end{align*} 
Provided that $|z|,|qz| < 1$, the generating function for the second type of sum defined above 
follows by differentiation of the 
geometric series function in the form of 
\begin{align*} 
S_p(z) & = \sum_{n \geq 0} S_{2,p,n} z^n = 
    \frac{1}{1-z} \times \sum_{n \geq 0} \binom{n}{p} p! (qz)^n - 
    \frac{p! q^p}{(1-q)^{p+1} (1-z)} \\ 
    & = \frac{z^p}{1-z} \times \frac{d^{(p)}}{dz^{(p)}}\left[\frac{1}{1-qz}\right] - 
    \frac{p! q^p}{(1-q)^{p+1} (1-z)} \\ 
\tag{iii} 
    & =
    \frac{p! (qz)^p}{(1-z) (1-qz)^{p+1}} 
     - \frac{p! q^p}{(1-q)^{p+1} (1-z)}. 
\end{align*} 
Next, we can compute using Carsten Schneider's \texttt{Sigma} package for 
\textit{Mathematica}\footnote{ 
     Available for non-commercial use on the 
     \emph{RISC} software group website: 
     \url{https://www3.risc.jku.at/research/combinat/software/Sigma/}. 
} 
that the first sums defined above satisfy the following homogeneous recurrence relation: 
\begin{align*} 
\tag{iv} 
 & \left(n(p+1)q-p(p+1)q\right) S_{1,p,n} + \left(n(q-1)+p-2q-2pq\right) S_{1,p+1,n} \\ 
 & \phantom{\quad\ } + 
     (1-q) S_{1,p+2,n} = 0. 
\end{align*} 
We complete the proof by showing that the sums, $S_{2,p,n}$ also satisfy the 
recurrence relation in (iv) for all integers $p \geq 1$, and that each of these 
sums produce the same formulas in $q$ and $n$ for the first few cases of $p \geq 1$. 
To show that the second sums satisfy the recurrence in (iv), we perform the 
following computations using the generating functions for the 
sequence in (iii) with \emph{Mathematica}: 
\begin{align*} 
\left((p+1)qz D-p(p+1)q\right) \cdot S_p(z) & + 
     \left((q-1)z D + p-2q-2pq\right) \cdot S_{p+1}(z) \\ 
     & + 
     \left(1-q\right) \cdot S_p(z) = 0. 
\end{align*} 
Finally, we compute the first few special cases of these sums explicitly 
using symbolic summation in the forms of 
\begin{align*} 
S_{1,1,n} = S_{2,1,n} & = \frac{q^{n+1}}{(1-q)^2}\left(qn-(n+1)\right) \\ 
S_{1,2,n} = S_{2,2,n} & = \frac{q^{n+1}}{(1-q)^3}\left(-n(n-1)q^2+2(n+1)(n-1)q 
     -(n+1)n\right) \\ 
S_{1,3,n} = S_{2,3,n} & = \frac{q^{n+1}}{(1-q)^4}\bigl(n(n-1)(n-2)q^3-3(n+1)(n-1)(n-2)q^2 \\ 
     & \phantom{=\frac{q^{n+1}}{(1-q)^4}\bigl(\ } + 
     3(n+1)n(n-2) q -(n+1)n(n-1)\bigr). 
\end{align*} 
Thus we have shown that the two sums satisfy the same homogeneous recurrence 
relations and initial conditions, and so must be equal for all $p, n \geq 1$, hence 
completing the proof of our result. 
\end{proof} 

\begin{proof}[Proof of Lemma~\ref{lemma_formulas_for_special_series_coeffs}: 
              A More General Result] 
\label{lemma_formulas_for_special_series_coeffs-ProofRef} 
The proofs of all three results are almost identical. 
We prove a more general result, of which the three identities are special cases. 
Namely, let $G(q) := \sum_{n} g_n q^n$ formally denote the ordinary generating 
function of an arbitrary sequence, and suppose that $m \geq 0$. Then we claim that 
\begin{align} 
\tag{iv} 
[q^x] q^s D^{(s)}\left[\frac{G(q)}{(1-q)^{m+1}}\right] 
     &  = 
     \sum_{r=0}^s \sum_{k=0}^x 
     \binom{s}{r} \binom{x-k+m}{s-r+m} \frac{(s-r+m)! k!}{m! \cdot (k-r)!} g_k. 
\end{align} 
To prove the claim, we require each of the results listed below 
where $f(q)$ and $g(q)$ denote functions which are each differentiable up to 
order $s$ for some integers $s, m \geq 0$. 
\begin{align} 
\label{eqn_lemma_SCCoeffs_LList_of_results} 
q^s \cdot D^{(s)}\left[f(q) g(q)\right] & = \sum_{r=0}^s \binom{r}{s} 
     q^{r} f^{(r)}(q) \times q^{s-r} g^{(s-r)}(q) \\ 
\notag 
q^s D^{(s)}\left[\frac{1}{(1-q)^{m+1}}\right] & = 
     \frac{q^s \cdot (m+1+s)!}{m! \cdot (1-q)^{m+s+1}} \\ 
\notag 
[q^{x}] \frac{1}{(1-q)^{k+1}} & = \binom{x+k}{k} \\ 
\notag 
[q^x] q^s \cdot G^{(s)}(q) & = \frac{x!}{(x-s)!} \cdot g_x 
\end{align} 
Once we have \eqref{eqn_lemma_SCCoeffs_LList_of_results}, 
our generalized result claimed in (iv) follows simply as 
a matter of concatenation of the double sums. 
The cited results in the previous equations 
are either well-known coefficients of power series (as in the last two 
cases), are established formulas, or are trivial to prove by induction 
(as in the second series identity case). 
Finally, we see that the identity in (i) corresponds to the special case where 
$(g_n, m) := (\sigma_{\alpha}(n), 0)$; in (ii) to the special case where 
$(g_n, m) := (n^{\alpha-1}, 1)$; and in (iii) to the case of out claim from (iv) where 
$(g_n, m) := (n^{\alpha-1}, 0)$. 
\end{proof} 

\subsection{Proof of the key proposition from Section~\ref{subSection_prop_KeyProp_sthDerivsOfLSeriesExps}} 
\label{subSection_ProofOfKeyLSProp}

With the key lemmas from the previous subsection at our disposal, we are finally 
able to complete the more involved, technical proof of 
Proposition~\ref{prop_KeyProp_sthDerivsOfLSeriesExps}. 

\begin{proof}[Proof of Proposition~\ref{prop_KeyProp_sthDerivsOfLSeriesExps}] 
\label{prop_KeyProp_sthDerivsOfLSeriesExps-ProofRef} 
We begin by defining the following shorthand for the higher-order $s^{th}$ 
derivatives of inner summation terms in the 
expansion of our key Lambert series identity in \eqref{eqn_key_LSeries_exp_ident} 
(see Lemma~\ref{lemma_LambertSeriesTerm_Sumjim2_Exps}): 
\begin{align*} 
\Sum_{s,i}(q) & := \sum_{j=0}^{i-2} q^s D^{(s)}\left[ 
     \frac{(i-j-1) q^{i+j}}{(1-q^i)}\right]. 
\end{align*} 
If we then apply Lemma~\ref{lemma_sthDerivsOfLambertSeries} 
in combination with the generalized product 
rule identity stated in \eqref{eqn_lemma_SCCoeffs_LList_of_results}, we see that 
\begin{align} 
\notag 
\Sum_{s,i}(q) & = \sum_{j=0}^{i-2} \sum_{r=0}^s \sum_{p=0}^{s-r} 
     \sum_{m=0}^{s-r} \sum_{k=0}^m (i-(j+1)) \binom{s}{r} \gkpSI{s-r}{m} 
     \gkpSII{m}{k} \binom{s-r-k}{p} \frac{j!}{(j-r)!} \times \\ 
\label{eqn_proof_tag_Sumsi_last_expr_v1}
     & \phantom{=\sum\sum\sum\sum\sum\ } \times 
     \frac{(-1)^{s-r-k-p} k! \cdot i^m q^{(p+1)i+j}}{(1-q^i)^{p+1}}. 
\end{align} 
To simplify notation for the multiple sum formulas we obtain in the next 
few formulas, we define the shorthand notation for the next sums where we denote 
$\Sigma_k f$ to be the multiple sum defined by $\Sigma_k$ over the function $f$ 
taken such that $f$ has at least $k$ parameters indexed naturally by the 
summations in $\Sigma_k$. The usage below the next definitions 
makes clear what we mean by using this convention.  
\begin{align*} 
\Sigma_4 f & \quad :\longmapsto \quad 
     \sum_{r=0}^s \sum_{p=0}^{s-r} \sum_{m=0}^{s-r} \sum_{k=0}^m f(r, p, m, k) \\ 
\Sigma_5 f & \quad :\longmapsto \quad 
     \sum_{r=0}^s \sum_{p=0}^{s-r} \sum_{m=0}^{s-r} \sum_{k=0}^m \sum_{n=0}^{r+1} 
     f(r, p, m, k, n) \\ 
\Sigma_8 f & \quad :\longmapsto \quad 
     \sum_{r=0}^s \sum_{p=0}^{s-r} \sum_{m=0}^{s-r} \sum_{k=0}^m \sum_{n=0}^{r+1} 
     \sum_{w=0}^s \sum_{m_1=0}^{r+1-n} \sum_{m_2=0}^n 
     f(r. p, m, k, n, m_1, m_2). 
\end{align*} 
When we utilize the notation defined in the previous equations to express the 
last expansion of $\Sum_{s,i}(q)$ given in \eqref{eqn_proof_tag_Sumsi_last_expr_v1}, 
we can apply Lemma~\ref{lemma_SumOverjPowqj_SumIdents_v2} to sum over the previous 
index $j$. This procedure yields the following exact expression for these sums: 
\begin{align*} 
\Sum_{s,i}(q) & = \Sigma_4\ \binom{s}{r} \gkpSI{s-r}{m} 
     \gkpSII{m}{k} \binom{s-r-k}{p} 
     \frac{(-1)^{s-r-k-p} k! r! q^{(p+1)i+r}}{(1-q^i)^{p+1}} \times \\ 
     & \phantom{=\Sigma_4\ \ } \times 
     \left( 
     \frac{i^{m+1}}{(1-q)^{r+1}} - \frac{(r+1) i^m}{(1-q)^{r+2}} 
     \right) \\ 
     & \phantom{=\ } + 
     \Sigma_5\ \binom{s}{r} \gkpSI{s-r}{m} \gkpSII{m}{k} \binom{s-r-k}{p} 
     \frac{(-1)^{s-r-k-p+n+1} k! q^{(p+2)i+n-1}}{(1-q^i)^{s-r+1}} \times \\ 
     & \phantom{=\Sigma_5\ \ } \times 
     \left( 
     \textstyle{
     \binom{r}{n} \frac{(i-1)! i^{m+1}}{(i-2-r)! (i-1-r+n) (1-q)^{r+1}} - 
     \binom{r+1}{n} \frac{i! i^m}{(i-2-r)! (i-1-r+n) (1-q)^{r+2}}
     } 
     \right). 
\end{align*} 
The last key jump in writing the formula in the previous equation in the 
form claimed by the proposition is to expand and re-write the factorial terms 
involving $i$ in $\Sigma_5$ as a sum (or multiple sum) over powers of $i$. 
To perform this change in the formula, we appeal to the expansions of 
factorial function products by the Stirling numbers of the first kind 
given by 
\begin{align*} 
 & \frac{i!}{(i-1-r)! (i-r+n)} = \prod_{j=0}^{r-n-1} (i-j) \times 
     \prod_{j=r-n+1}^r (i-j) \\ 
     & \qquad = 
     \left(\sum_{m=0}^{r-n} \gkpSI{r-n}{m} (-1)^{r-n-m} i^m\right) \times 
     \left(\sum_{m=0}^{n} \gkpSI{n}{m} (-1)^{n-m} (i+n-1-r)^m\right) \\ 
     & \qquad = 
     \sum_{v=0}^r \sum_{m_1=0}^{r-n} \sum_{m_2=0}^n 
     \gkpSI{r-n}{m_1} \gkpSI{n}{m_2} \binom{m_2}{v-m_1} (-1)^{r+m_1+m_2}
     (n-1-r)^{m_1+m_2-v} i^v. 
\end{align*} 
We are now able to define equivalent forms of the 
summation functions depending on $q$ and $i$ in 
Definition~\ref{def_Sum4sqi_Sum8sqi_init_defs} as 
\begin{align*} 
\Sum_{4,s}(q, i) & = \Sigma_4\ \binom{s}{r} \gkpSI{s-r}{m} \gkpSII{m}{k} 
     \binom{s-r-k}{p} \frac{(-1)^{s-r-k-p} k! r! 
     q^{(p+1)i+r}}{(1-q^i)^{s-r+1}} \times \\ 
     & \phantom{=\Sigma_4\ \ } \times 
     \left( 
     \frac{i^{m+1}}{(1-q)^{r+1}} - \frac{(r+1) i^m}{(1-q)^{r+2}} 
     \right) \\ 
\Sum_{8,s}(q, i) & = \Sigma_8\ \binom{s}{r} \gkpSI{s-r}{m} \gkpSII{m}{k} 
     \binom{s-r-k}{p} \gkpSI{r+1-n}{m_1} \gkpSI{n}{m_2} \binom{m_2}{w-m-m_1} 
     \times \\ 
     & \phantom{=\Sigma_8\ \ } \times 
     \frac{(-1)^{s-k-p+n+m_1+m_2} k! 
     q^{(p+2)i+n-1} (n-2-r)^{m+m_1+m_2-w} i^w}{(1-q^i)^{s-r+1}} \times \\ 
     & \phantom{=\Sigma_8\ \ } \times 
     \left( 
     \binom{r}{n} \frac{1}{(1-q)^{r+1}} - \binom{r+1}{n} \frac{1}{(1-q)^{r+2}} 
     \right). 
\end{align*} 
The proof is completed by reconciling the notation from 
Definition~\ref{def_Multiple_Init_Coeff_Sums} 
together with an application of 
Lemma~\ref{lemma_formulas_for_special_series_coeffs} to the 
remaining term in the sum from \eqref{eqn_key_LSeries_exp_ident}. 
\end{proof} 

\section{Conclusions} 
\label{Section_Concl} 

The identities we have proved in the article 
are derived from elementary series transformations and 
operations on truncated partial sums of Lambert series generating functions. 
The resulting new identities provide exact formulas for the classical 
generalized sum-of-divisors 
functions, $\sigma_{\alpha}(n)$, expanded by polynomial multiples of the related 
bounded-index divisor sums defined by the functions, 
$\sigma_{\alpha,m}(n)$, in 
\eqref{eqn_BoundedIndex_DivisorFuncs_def_v2}. 
Based on experimental computational procedures and inspection, 
there do not appear to be correspondingly elegant identities that can be derived 
from operations in more general Lambert series expansions over other 
multiplicative functions, $f(i)$, of the form 
\[
\widehat{L}_{f}(q) := \sum_{n \geq 1} \frac{f(n) q^n}{1-q^n} = \sum_{m \geq 1} \left( 
     \sum_{d|m} f(d) \right) q^m, |q| < 1, 
\]
when $f(n) \neq n^{\alpha}$ for some $\alpha \geq 0$. 
We can express the formulas in 
\eqref{eqn_TauRhosxkx_GenFormPolyCoeffExps} 
for positive real $\beta > 0$ by the symmetric negative-order identity for the 
generalized sum of divisors functions which shows that 
$\sigma_{-\beta}(n) = n^{-\beta} \cdot \sigma_{\beta}(n)$. 
Non-trivial simplifications of the bounded-index divisor sum functions, 
$\sigma_{\beta,m}(n)$, are less obvious to state for increasing integer 
parameters $m \geq 3$. 

\subsection{Future directions and topics to consider} 

One possible interpretation of the sums we have obtained 
is to formulate how the divisors corresponding to the left-hand-side 
expansions of $\sigma_{\alpha}(n)$ are partitioned by the modified 
divisor sum functions and the observed polynomial multiples of these functions. 
For example, consider the following variant of the bounded-index divisor functions 
defined in \eqref{eqn_BoundedIndex_DivisorFuncs_def_v2} for integer parameters $1 \leq m_1 < m_2 \leq n$: 
\[
\sigma_{\alpha,m_1,m_2}(n) := \sum_{\substack{d|n \\ 
     m_1 < d \leq m_2}} d^{\alpha}, 
     \alpha \in \mathbb{C}. 
\] 
For natural numbers $n \geq 1$, let the sequences 
$1=r_1^{(n)} < r_2^{(n)} < \cdots < r_{d(n)}^{(n)} = n$ denote the distinct divisors of 
$n$ in ascending order. 
It follows that we can partition the terms in the summation-based form of $\sigma_{\alpha}(n)$ into 
disjoint sets of divisors using the formula: 
\[
\sigma_{\alpha}(n) = \sum_{j=2}^{d(n)} \sigma_{\alpha,r_{j-1}^{(n)},r_{j}^{(n)}}(n). 
\]
We can also prove the following identity: 
\[
\sum_{m=1}^{n} \sigma_{\alpha,m}(n) = n \times \sum_{d|n} d^{\alpha-1} = n \cdot \sigma_{\alpha-1}(n).  
\] 
The last formula is perhaps more suggestive of the types of partition-related identities that we 
might try to invoke on the new expressions for $\sigma_{\alpha}(n)$ proved in this article. 

\renewcommand{\refname}{References}


\begin{thebibliography}{10}\footnotesize

\bibitem{JNT-NEWCVLSDN} 
C. Ballantine and M. Merca, {New convolutions for the number of divisors}, 
        {\it J. Number Theory} {\bf 170} (2017), 17--34. 

\bibitem{SEGLADUN-GFS} 
S. E. Gladun, {A generating function for $\sigma(3n-1)$}, 
        {\it Math. Notes} {\bf 95} (2014), 565--569. 

\bibitem{GKP}
R.~Li. Graham, D.~E. Knuth, and O.~Patashnik,
    {\it Concrete Mathematics: A Foundation for Computer Science},
    Addison-Wesley, Boston, 1994.

\bibitem{HARDY-WRIGHT} 
G. H. Hardy and E. M. Wright, {\it An Introduction to the Theory of Numbers}, 
  Oxford University Press, Oxford, 2008. 

\bibitem{JNT-ANEWLOOK} 
M. Merca, {A new look on the generating function for the number of divisors}, 
        {\it J. Number Theory} {\bf 149} (2015), 57--69. 

\bibitem{NISTHB}
F.~W.~J. Olver, D.~W. Lozier, R.~F. Boisvert, and C.~W. Clark, {\it
  {NIST} Handbook of Mathematical Functions}, 
  Cambridge University Press, Cambridge, 2010.

\bibitem{ELEMEVALCVLSUMS2002} 
J. G. Huard, et al. Elementary evaluation of certain convolution sums 
        involving divisor functions, in {\it Number Theory for the Millennium II}, edited by 
        M. A. Bennett, B. C. Berndt, N. Boston, H. G. Diamond, A. J. Hildebrand, and W. Philipp, A. K. Peters (2002), 
        229-274, 
        \url{http://people.math.carleton.ca/~williams/papers/pdf/249.pdf}. 

\bibitem{ONARITHFNS-RAMANUJAN} 
S. Ramanujan, On certain arithmetical functions, {\it Messenger Math.} {\bf 45} (1916), 11-15. 

\end{thebibliography}
\end{document}